\documentclass[11pt,fleqn,leqno]{article} 
\usepackage[final]{graphicx}
\usepackage[cmex10]{amsmath} 
\usepackage{amsthm}
\usepackage{amssymb}
\usepackage{amsfonts}
\usepackage{setspace}
\usepackage{natbib}
\usepackage[final,colorlinks]{hyperref}
\usepackage[margin=1in]{geometry}

\DeclareMathOperator{\Var}{Var}

\newcommand{\abs}[1]{\left\lvert#1\right\rvert}
\newcommand{\norm}[1]{\left\lVert#1\right\rVert}
\newcommand{\Reals}{\mathbb{R}}
\newcommand{\transpose}{^{\top}}

\newtheoremstyle{nonum}{}{}{\itshape}{}{\bfseries}{.}{ }{\thmnote{#3}}
\theoremstyle{nonum}
\newtheorem{thm}{}
\newtheorem{lemma}{}

\begin{document}
\title{A simple and general proof for the convergence of Markov processes to their mean-field limits}
\author{Benjamin Armbruster\footnote{Email address:  armbrusterb@gmail.com}}
\date{\today}

\maketitle

\begin{abstract}
Mean-field models approximate large stochastic systems by simpler differential equations that are supposed to approximate the mean of the larger system.  It is generally assumed that as the stochastic systems get larger (i.e., more people or particles), they converge to the mean-field models.  Mean-field models are common in many fields, but their convergence is rarely proved.  The existing approaches rely on operator semigroups, martingales, PDEs, or infinite systems of ODEs.  We give a general proof for their convergence using only Taylor's theorem and basic ODE results.  We hope this allows applied researchers to routinely show convergence of their mean-field models, putting their work on a stronger foundation.
\end{abstract}

\section{Introduction}
Mean-field approximations describe the behavior of large stochastic systems using simpler deterministic differential equations.  Mean-field limits are used in many fields from population biology, the physical sciences, to queuing models.  They are also known as ``fluid limits'' in the context of queueing models and ``mean-field approximations'' when convergence is not the focus.  Intuitively, the concept of a mean-field limit is an extension of the law of large numbers from random variables to stochastic processes.

We start by giving a precise description of the mean-field limit.  One approach is to define the stochastic process $\tilde{X}_n(t)$ describing the state of $n$ individual, usually identical but interacting, particles, $\tilde{X}_n(t):=(\tilde{X}_{n,1}(t),\dotsc,\tilde{X}_{n,n}(t))$, where $\tilde{X}_{n,i}(t)$ denotes the state of the $i$th particle.  Then the mean-field limit is the deterministic process $\bar{x}(t)$, to which the average behavior $X_n(t):=(1/n)\sum_{i=1}^n \tilde{X}_{n,i}(t)$ converges: $X_n(t)\to \bar{x}(t)$ as $n\to\infty$.  If the particles are not interacting (i.e., the $\tilde{X}_{n,i}$ are independent of each other), then this is a simple consequence of the law of large numbers.  Two recent papers, \citep{BHL2013} and \citep{BL2008}, in this stream of literature make the less trivial assumption that any particle is affected by the others, but only through the number of other particles in each state, and not their specific identities.  The papers give an overview of mean-field models and prove convergence results using time-change arguments and techniques from martingales and stochastic approximation algorithms.

We instead will take a slightly more general approach, that is not limited to processes describing the collective state of $n$ particles.  Instead we directly consider the conditions for a sequence of Markov processes to converge to a deterministic limit $X_n(t)\to \bar{x}(t)$ as $n\to\infty$.  Let $\Reals^k$ contain the state space of $X_n$ and we will write $X_{n,i}$ to denote component $i$.  Also let $Q_n$ be the transition rate matrix of $X_n$, defined so that $Q_n(x,\hat{x})$ is the instantaneous rate at which $X_n$ jumps from $x$ to $\hat{x}$, or equivalently, $(1/\epsilon)P[X_n(t+\epsilon)=\hat{x}|X_n(t)=x]\to Q_n(x,\hat{x})$ as $\epsilon\to 0$.  While for simplicity we focus on time-homogenous Markov processes with discrete state spaces, our theory works just as well for time-inhomogeneous Markov processes with continuous state spaces.

Usually, the master equations (also known as the forward equations), a set of differential equations, one for each $x$ in the state space of $X_n$,
\[
	P[X_n(t)=x]' = \sum_{\hat{x}}P[X_n(t)=\hat{x}]Q_n(\hat{x},x),
\]
are the starting point when analyzing a Markov process.
We however will start with the Kolmogorov (backward) equation, which describes the dynamics for an arbitrary expectation of a Markov process:
	\[E[a(X_n(t))]'=\sum_{\hat{x}}P[X_n(t)=\hat{x}]'a(\hat{x})=\sum_{x,\hat{x}}P[X_n(t)=x]Q_n(x,\hat{x})a(\hat{x})=E[c(X_n(t))],\]
where $c(x):=\sum_{\hat{x}}Q_n(x,\hat{x})a(\hat{x})$.  Since by definition $Q_n(x,x)=-\sum_{\hat{x}\neq x}Q_n(x,\hat{x})$, we can also write $c(x)$ more intuitively as $c(x)=\sum_{\hat{x}\neq x}Q_n(x,\hat{x})(a(\hat{x})-a(x))$, where we add the products of the transition rates and the size of the resulting changes in $a(X)$.  This allows us to describe the behavior of the first moment of $X_n$ 
and the sum of the second moments of $X_n$:
\begin{align*}
	E[X_n(t)]'&=E[m_{1,n}(X(t))],\\
	E\Bigl[\sum_i X_{n,i}(t)^2\Bigr]'&=E[m_{2,n}(X(t))],
\end{align*}
where we define $m_{1,n}(x):=\sum_{\hat{x}}Q_n(x,\hat{x})\hat{x}$ and $m_{2,n}(x):=\sum_{\hat{x}}Q_n(x,\hat{x})\sum_i \hat{x}_i^2$, or equivalently 
\begin{align}
	m_{1,n}(x)&=\sum_{\hat{x}\neq x}Q_n(x,\hat{x})(\hat{x}-x),\nonumber\\
\label{eq1}	m_{2,n}(x)&=\sum_{\hat{x}\neq x}Q_n(x,\hat{x})\Bigl(\sum_i\hat{x}_i^2-\sum_i x_i^2\Bigr).
\end{align}
We will need to be careful about keeping track of the various indices.  We will write $m_{1,n,i}$ to denote the $i$th component of $m_{1,n}$.

Let us assume that the mean-field limit, $\bar{x}(t)$, solves a deterministic initial value problem,
\begin{equation}
\label{eq2}	x'=\bar{m}_1(x),\quad x(0)=x_0.
\end{equation}
Here $\bar{m}_{1,i}(x)$ refers to the $i$th component of the vector $\bar{m}_1(x)$ describing the mean-field dynamics.  Natural conditions for the convergence of $X_n(t)\to \bar{x}(t)$ as $n\to\infty$ are $X_n(0)\to x_0$, $m_{1,n}\to\bar{m}_1$, and that the mean-field equation \eqref{eq2} has a unique solution.

Density dependent processes are an example where the limit $m_{1,n}\to\bar{m}_1$ holds naturally.  In that case, we have a process $\tilde{X}_n(t)$ of $n$ particles on $[0,n]^k$.  This process makes jumps of different sizes, where a jump of size $\Delta_i$ from $\tilde{x}$ to $\tilde{x}+\Delta_i$ occurs at a rate $n q_i(\tilde{x}/n)$, which depends on the density, $\tilde{x}/n$, and is proportional to the system size, $n$.  If we examine the scaled process, $X_n(t):=\tilde{X}_n(t)/n$, we find that the mean behavior at $x=\tilde{x}/n$ is $m_{1,n}(x) = \sum_i n q_i(\tilde{x}/n)\Delta_i/n = \sum_i q_i(x)\Delta_i$ and does not depend on $n$.

In addition to the conditions on $X_n(0)$ and $m_{1,n}$, we need conditions ensuring that $\lim X_n$ is deterministic (a random walk converging to Brownian motion is an example where the limit is not deterministic).  Equivalent conditions for $\lim X_n$ being deterministic include, $\sum_i\Var[\lim_n X_{n,i}]=0$; or equivalently $\sum_i E[(\lim_n X_{n,i})^2]-E[\lim_n X_{n,i}]^2=0$.  Taking the derivative of the last condition suggests the following lemma.

\begin{lemma}[Lemma 0]\label{lem0}
If $Y(0)$ is deterministic and $E[\sum_i Y_i(t)^2]'=\sum_i 2E[Y_i(t)]'E[Y_i(t)]$ for $t\leq T$, then $Y(t)$ is deterministic for $t\leq T$.
\end{lemma}
\begin{proof}  Obvious.  \end{proof}

This suggests the additional natural condition
\[ m_{2,n}\to \sum_i 2\bar{x}_i'(t)\bar{x}_i(t)
=:\bar{m}_{2}(\bar{x}(t)) \text{ as $n\to\infty$}.\]
Our main theorem then merely adds the required technical details.

\begin{thm}[Theorem 1]\label{thm1} Suppose that the Markov process $X_n(t)$ is on some bounded convex set $C\subseteq\Reals^k$;
\begin{subequations}
\begin{align}
\label{eq3a}	E[X_n(t)]'&=E[m_{1,n}(X_n(t))];\\
\label{eq3b}	E\Bigl[\sum_i X_{n,i}(t)^2\Bigr]'&=E[m_{2,n}(X_n(t))];
\end{align}
$X_n(0)\to x_0$ in mean-square; $m_{1,n}\to \bar{m}_1$ uniformly; and $m_{2,n}\to \bar{m}_2$ uniformly, where we define
\begin{equation}
\label{eq3c}	\bar{m}_{2}(x):=\sum_i 2\bar{m}_{1,i}(x)x_i.
\end{equation}
\end{subequations}
We further assume that the above functions are defined on $C$; $\bar{m}_1$ is Lipschitz; and $m_{1,n}''$ and $m_{2,n}''$ are continuous and bounded independently of $n$.  Then $X_n(t)$ converges uniformly in mean-square on finite time intervals to the solution, $\bar{x}(t)$, of the mean-field equations, $x'=\bar{m}_1(x)$, $x(0)=x_0$.
\end{thm}

Recall that the convergence of a random vector to a constant in mean-square, $Y_n\to y$, means that $E[\norm{Y_n-y}^2]\to 0$, or equivalently that the first and second moments converge, $E[Y_n]\to y$ and $E[Y_{n,i}Y_{n,j}]\to y_iy_j$ for any $i,j$.  Uniform convergence on finite time intervals means that $\sup_{t\leq T}\norm{X_n(t)-\bar{x}(t)}\to 0$.  Note that the Lipschitz condition on $\bar{m}_1$ ensures the existence and uniqueness of the mean-field solution.  Also note that $m_{1,n,i}''(x)$ and $m_{2,n}''(x)$ are $k\times k$ Hessian matrices.

The first to prove such results for the convergence of Markov processes to deterministic differential equations was Kurtz in 1970 \citep{Kurtz1970}.  That proof used advanced tools from operator semigroup theory.  The follow-up paper used martingale techniques \citep{Kurtz1971} (see also the book \citep{EthierKurtz}).  The main contribution of this paper is to provide a short proof using only elementary techniques.  A smaller contribution is that we prove convergence in mean-square which is slightly stronger than the convergence in probability proved by Kurtz.

Our approach using simple tools from ODE theory, extends previous work by the author and Beck \citep{ABSIS,BASIR}.  In particular our differential equation for the sum of second moments \eqref{eq3c} is similar to the differential equation for the sum of the variances in \cite{BASIR}. The previous work is specific to the susceptible-infected-susceptible (SIS) and susceptible-infected-recovered (SIR) models.  It also relies on the standard ODE result for the continuous dependence on parameters but uses other moment bounds instead of one derived from a simple Taylor expansion, which we use here.  For the special cases of SIS and SIR models, there are also proofs using PDEs \citep{STK2011} and infinite systems of ODEs \citep{SK2012}. 

In Section~\ref{sec:proof} we prove \nameref{thm1}.  Then in Section~\ref{sec:example}, we go over our proof on a simple example, before showing how to apply it to a non-trivial one.  We conclude in Section~\ref{sec:conclusion}.
 
\section{Proof of Main Theorem}\label{sec:proof}
For any symmetric matrix $A$, let $\rho(A)$ be its spectral radius.

\begin{lemma}[Lemma 2]\label{lem2}
Suppose $Y$ is a random variable on some convex set $C\subseteq\Reals^k$, and $f''$ is defined on $C$, continuous, and $\rho(f''(x))\leq b$.  Then $E[f(Y)]=f(E[Y])+h\sum_i \Var[Y_i]$, where $\abs{h}\leq b/2$.
\end{lemma}
\begin{proof}
This is a simple extension of Taylor's theorem with remainder.  We now consider $f(\cdot)$ along the line between $Y$ and $E[Y]$ and apply the scalar Taylor's theorem with the Lagrange form of the remainder around the value $E[Y]$,
\[
	f(Y)=f(E[Y])+f'(E[Y])\transpose (Y-E[Y])+(1/2)(Y-E[Y])\transpose f''(\tilde{Y}) (Y-E[Y]),\] 
for some $\tilde{Y}$ between $Y$ and $E[Y]$ (that depends on the value of $Y$).  
Rearranging,
\[
	f(Y)-f(E[Y])-f'(E[Y])\transpose (Y-E[Y])=(1/2)(Y-E[Y])\transpose f''(\tilde{Y}) (Y-E[Y]).
\]
Taking expectations, the $Y-E[Y]$ term on the left hand side vanishes,
\begin{equation}\label{eq:taylor}
 E[f(Y)]-f(E[Y])=(1/2)E\Bigl[(Y-E[Y])\transpose f''(\tilde{Y}) (Y-E[Y])\Bigr].
\end{equation}
Since $\rho(f''(\tilde{Y}))\leq b$, we can bound the quadratic term: 
\[ \abs{(Y-E[Y])\transpose f''(\tilde{Y}) (Y-E[Y])}\leq b\norm{Y-E[Y]}_2^2 = b\sum_i (Y_i-E[Y_i])^2.\]
Since this bound holds for every realization of $Y$, it also holds in expectation:
\begin{multline*}
 \abs{E\bigl[(Y-E[Y])\transpose f''(\tilde{Y}) (Y-E[Y])\bigr]}
 \leq E\left[\abs{(Y-E[Y])\transpose f''(\tilde{Y}) (Y-E[Y])}\right]\\
 \leq b\sum_i E[(Y_i-E[Y_i])^2] = b\sum_i \Var[Y_i].
\end{multline*}
Taking the absolute value of both sides of \eqref{eq:taylor} and then applying the bound,
\[ \abs{E[f(Y)]-f(E[Y])}\leq (b/2)\sum_i \Var[Y_i]. \]
Choosing $h=(E[f(Y)]-f(E[Y]))/\sum_i \Var[Y_i]$ proves the claim (and if $\sum_i \Var[Y_i]=0$, we set $h=0$).
\end{proof}

Suppose $\rho(m_{1,n,i}''(x)),\rho(m_{2,n}''(x))\leq b_2$.  Then we use \nameref{lem2} to define $h_{1,n,i}(t)$ and $h_{2,n}(t)$ so that
\begin{align*}
	E[m_{1,n,i}(X(t))]&=m_{1,n,i}(E[X(t)])+h_{1,n,i}(t)\sum_j \Var[X_j(t)],\\
	E[m_{2,n}(X_n(t))]&=m_{2,n}(E[X(t)])+h_{2,n}(t)\sum_j \Var[X_j(t)],
\end{align*}
and $\abs{h_{1,n,i}(t)},\abs{h_{2,n}(t)}\leq b_2/2$.  We then define $z_n:=(E[X_n],\sum_j E[X_{n,j}^2])$, which solves the initial value problem, $z'=g_n(t,z)$ and $z(0)=z_{0,n}$, where
\begin{align*}
	g_{1,n,i}(t,z)&:=m_{1,n,i}(z_1)+h_{1,n,i}(t)\Bigl(z_2-\sum_j z_{1,j}^2\Bigr),\\
	g_{2,n}(t,z)&:=m_{2,n}(z_1)+h_{2,n}(t)\Bigl(z_2-\sum_j z_{1,j}^2\Bigr),\\
	z_{0,n}&:=\bigl(X_n(0),\sum_j X_{n,j}(0)^2\bigr).
\end{align*}
Note that $z_2-\sum_j z_{1,j}^2 = \sum_j \Var[X_j]$. 
We will compare this to the initial value problem, $z'=\bar{g}_n(t,z)$ and $z(0)=\bar{z}_0$, where
\begin{align*}
	\bar{g}_{1,n,i}(t,z)&:=\bar{m}_{1,i}(z_1)+h_{1,n,i}(t)\Bigl(z_2-\sum_j z_{1,j}^2\Bigr),\\
	\bar{g}_{2,n}(t,z)&:=\bar{m}_{2}(z_1)+h_{2,n}(t)\Bigl(z_2-\sum_j z_{1,j}^2\Bigr),\\
	\bar{z}_0&:=\bigl(x_0,\sum_j x_{0,j}^2\bigr),
\end{align*}
Here $(\bar{x},\sum_j \bar{x}_j^2)$ is a solution because it makes the term for the sum of variances, $z_2-\sum_j z_1^2$, vanish, and $(\sum_j \bar{x}_j^2)'=\bar{m}_2(\bar{x})$ due to \eqref{eq3c}.

\begin{lemma}[Lemma 3]\label{lem3}
Consider the initial value problems $x'=f_1(t,x)$, $x(0)=x_1$ and $x'=f_2(t,x)$, $x(0)=x_2$ with solutions $\phi_1(t)$ and $\phi_2(t)$ respectively.  If $f_1$ is Lipschitz in $x$ with constant $L$; $\norm{x_2-x_1}\leq \Delta$; and $\norm{f_1(t,x)-f_2(t,x)}\leq M$, then $\norm{\phi_1(t)-\phi_2(t)}\leq (\Delta+M/L)e^{Lt}-M/L$.
\end{lemma}
\begin{proof}  This is a standard result about the continuous dependence of ODE solutions on parameters and initial conditions.  We give a proof in the Appendix.  \end{proof}

Using the definition \eqref{eq3c}, $\bar{m}_2(z_1)$ is Lipschitz because $z_1$ is bounded (i.e., $C$ is bounded) and each $\bar{m}_{1,i}$ is Lipschitz.  Since, $\bar{m}_1$ and $\bar{m}_2$ are Lipschitz and $h_{1,n}$, $h_{2,n}$, and $z_1$ are bounded, then $\bar{g}_n(t,z)$ is Lipschitz in $z$ with a Lipschitz constant independent of $n$.
Let $M_n:=\norm{g_n(t,z)-\bar{g}_n(t,z)}$ and $\Delta_n:=\norm{z_{0,n}-\bar{z}_0}$.  This allows us to apply \nameref{lem3}, which for finite time intervals bounds $\norm{z_n-(\bar{x},\sum_j \bar{x}_j^2)}$.  Since $\Delta_n\to0$ and uniform convergence implies 
\[ M_n=\norm{(m_{1,n}(x),m_{2,n}(x))-(\bar{m}_1(x),\bar{m}_2(x))}\to 0,\]
the bound from \nameref{lem3} goes to zero, proving that $z_n\to (\bar{x},\sum_j \bar{x}_j^2)$ uniformly on finite time intervals.  This proves the theorem.

\section{Example}\label{sec:example}
We now consider two examples.  In the first example we use the the simple, scalar, SIS model to review the steps of the proof.  The second example is a more complicated epidemic model, where we show how easy it is to check the conditions of \nameref{thm1}.

\subsection{Simple SIS Model}
In the susceptible-infected-susceptible (SIS) model, we have a population of $n$ individuals, of which $\tilde{X}_n(t)$ are infected.  This is a birth-death process, where given $\tilde{x}$ infected nodes, an additional infection occurs (increasing $\tilde{x}$ by one) at a rate $\beta \tilde{x}(1-\tilde{x}/n)$ and an infection recovers (decreasing $\tilde{x}$ by one) at a rate $\gamma \tilde{x}$.  These can be written as density dependent rates, $n q_+(\tilde{x}/n)$ and $n q_-(\tilde{x}/n)$ for an increase and decrease respectively, with $q_+(x)=\beta x(1-x)$ and $q_-(x)=\gamma x$.

Our first step is to define the scaled process, $X_n(t)=\tilde{X}_n(t)/n$, on $[0,1]$, which increases and decreases by $1/n$ at rates $n q_+(x)$ and $n q_-(x)$ respectively.  From the definition of $m_{1,n}$, 
\begin{equation}\label{eq:m1SIS}
	m_{1,n}(x) = +(1/n)(nq_+(x))+(-1/n)(nq_-(x))
		=\beta x(1-x) - \gamma x
		=(\beta-\gamma) x - \beta x^2.
\end{equation}
Since $m_{1,n}$ is already quadratic, the application of \nameref{lem2}, involves merely taking the expectation of $m_{1,n}(X)$ and substituting $E[X^2]=E[X]^2+\Var[X]$:
\[ E[m_{1,n}(X)] = m_{1,n}(E[X]) - \beta \Var[X]
= m_{1,n}(E[X]) + h_{1,n} \Var[X], \]
where $h_{1,n}=-\beta$.

Turning to the second moment, we start with the definition,
\begin{align}
	m_{2,n}(x) =& nq_+(x)((x+1/n)^2-x^2) + nq_-(x)((x-1/n)^2-x^2)\nonumber\\
	 =& 2x(\beta x(1-x)-\gamma x) + (1/n)(\beta x(1-x)+\gamma x)\label{eq:m2SIS}\\
	 =& (\beta/n+\gamma/n) x + (2\beta-2\gamma-\beta/n)x^2 - 2x^3.\nonumber
\end{align}
Then $m''_{2,n}(x) = 2(2\beta-2\gamma-\beta/n) + 12x$, and since $x\in[0,1]$ we have the bound $\abs{m''_{2,n}(x)}\leq 2\abs{2\beta-2\gamma-\beta/n} + 12$.  Applying \nameref{lem2},
\[ E[m_{2,n}(X)] = m_{2,n}(E[X])+h_{2,n}\Var[X], \]
for some $\abs{h_{2,n}}\leq \abs{2\beta-2\gamma-\beta/n} + 6$.
Essentially what we did was the following: (1) apply a first order Taylor expansion with remainder to $x^3$ to derive the (loose but valid) bound $0\leq x^3\leq 3x^2$ for $x\in[0,1]$; (2) apply this bound to $m_{2,n}$ to derive a quadratic right hand side; and then take expectations and apply again the substitution $E[X^2]=E[X]^2+\Var[X]$.

Applying the reverse substitution, $\Var[X]=E[X^2]-E[X]^2$, to our equations for $E[m_{1,n}(X)]$ and $E[m_{2,n}(X)]$ we arrive at a closed system of two differential equations for the first two moments: $E[X_n(t)]$ and $E[X_n(t)^2]$:
\begin{align*}
	E[X_n(t)]'=E[m_{1,n}(X_n(t))]=m_{1,n}(E[X_n(t)]) 
	- \beta (E[X_n(t)^2]-E[X_n(t)]^2)\\
	E[X_n(t)^2]'=E[m_{2,n}(X_n(t))]
	=m_{2,n}(E[X_n(t)]) 
	+ h_{2,n}(t) (E[X_n(t)^2]-E[X_n(t)]^2)
\end{align*}
From the definition, $\bar{m}_2:=2x\bar{m}_1(x)$; since the right hand side of \eqref{eq:m1SIS} does not depend on $n$, $m_{1,n}=\bar{m}_1$; and $m_{2,n} \to \bar{m}_2$ due to \eqref{eq:m2SIS}.  Then as the right hand sides of the differential equations converge to $\bar{m}_1$ and $\bar{m}_2$, so do their solutions.

Comparing this approach to that taken in \cite{ABSIS} or \cite{BASIR}, we see that the key step in all these approaches is bounding the higher order terms (beyond the quadratic term).  While the particular inequalities used in those papers are tighter, here we use Taylor's theorem with remainder, which is more generally applicable.

\subsection{More Complicated Epidemic Model}
Consider a susceptible-infected SI model for the spread of an incurable sexually transmitted disease in a heterosexual population.  Consider a population of $n$ identical males and $n$ identical females in the susceptible or infected state.  Thus, our system state is $\tilde{X}_n=(\tilde{X}_{n,\mathrm{ms}},\tilde{X}_{n,\mathrm{mi}},\tilde{X}_{n,\mathrm{fs}},\tilde{X}_{n,\mathrm{fi}})$, the count of the number of male-susceptible, male-infected,  female-susceptible, and female-infected individuals, respectively.  The two possible transitions are a female infecting a male (reducing $\tilde{X}_{n,\mathrm{ms}}$ by one and increasing $\tilde{X}_{n,\mathrm{mi}}$ by one) at a rate $\beta_{\mathrm{m}} \tilde{X}_{n,\mathrm{fi}}\tilde{X}_{n,\mathrm{ms}}/n$ and a male infecting a female at a rate $\beta_{\mathrm{f}}\tilde{X}_{n,\mathrm{mi}}\tilde{X}_{n,\mathrm{fs}}/n$.  Such a model suits our purposes since it is simple enough to use as an example, yet more complicated than the well known SIR model.

Our normalized Markov process is $X_n(t):=\tilde{X}_n(t)/n$ on $[0,1]^4$.  It is easy to check that 
\begin{align*}
m_{1,n,\mathrm{ms}}(x)&=\bar{m}_{1,\mathrm{ms}}(x)=-\beta_{\mathrm{m}} x_{\mathrm{fi}}x_{\mathrm{ms}},\quad&m_{1,n,\mathrm{mi}}(x)&=\bar{m}_{1,\mathrm{mi}}(x)=\beta_{\mathrm{m}} x_{\mathrm{fi}}x_{\mathrm{ms}},\\
m_{1,n,\mathrm{fs}}(x)&=\bar{m}_{1,\mathrm{fs}}(x)=-\beta_{\mathrm{f}} x_{\mathrm{mi}}x_{\mathrm{fs}},\quad& m_{1,n,\mathrm{fi}}(x)&=\bar{m}_{1,\mathrm{fi}}(x)=\beta_{\mathrm{f}} x_{\mathrm{mi}}x_{\mathrm{fs}}.
\end{align*}
Using definition \eqref{eq1},
\begin{align*}
	m_{2,n}(x)=&n\beta_{\mathrm{m}} x_{\mathrm{fi}}x_{\mathrm{ms}}
	\bigl(((x_{\mathrm{ms}}-1/n)^2-x_{\mathrm{ms}}^2)
	+((x_{\mathrm{mi}}+1/n)^2-x_{\mathrm{mi}}^2)\bigr)\\
	&+n\beta_{\mathrm{f}} x_{\mathrm{mi}}x_{\mathrm{fs}}
	\bigl(((x_{\mathrm{fs}}-1/n)^2-x_{\mathrm{fs}}^2)
	+((x_{\mathrm{fi}}+1/n)^2-x_{\mathrm{fi}}^2)\bigr),
\end{align*}
and applying the identity $n(x\pm 1/n)^2-x^2=\pm 2x+1/n$,
\[
	m_{2,n}(x)=2\beta_{\mathrm{m}} x_{\mathrm{fi}}x_{\mathrm{ms}}(-x_{\mathrm{ms}}+x_{\mathrm{mi}}+1/n)
	+2\beta_{\mathrm{f}} x_{\mathrm{mi}}x_{\mathrm{fs}}(-x_{\mathrm{fs}}+x_{\mathrm{fi}}+1/n),
\]
makes it is easy to see that these equations satisfy the conditions of \nameref{thm1}.

\section{Discussion}\label{sec:conclusion}
To summarize our approach, we relied on the Kolmogorov backward equations to create differential equations for expectations of the stochastic process.  However, on the right hand sides of the equations were expectations of functions of the state (rather than functions of the expected state).  We approximated those using Taylor's theorem to derive a closed system of differential equations for the first two moments of the stochastic process.  Then we applied a standard approximation result for the continuous dependence on parameters of ODEs.

With the exception of the approximation result, these are all techniques seen in undergraduate courses.  The author hopes that the easy to check conditions of these results will allow applied users of mean-field models in fields such as epidemiology to routinely prove convergence of their results, and that the simple techniques involved in their proofs will inspire further research in the intersection of stochastic processes and ODEs. 

One limitation of our technical conditions is that they cannot be used to prove the standard fluid limits for queues.  There are two reasons: first the state space (i.e., the number in the queue) is unbounded and second, the right hand side of the mean-field equation, $\bar{m}_1(x)=-\mu I(x>0)$, is not Lipschitz since it has a discontinuity at zero (here $I(\cdot)$ is the indicator function).  Relaxing those technical conditions is an area for further research.

\paragraph{Acknowledgements.}  The author thanks Ekkehard Beck and Simon P\'eter for helpful comments.

\bibliography{refs}

\section*{Appendix}

The following is a version of the standard ODE result on the continuous dependence on parameters and initial conditions. 

\begin{lemma}[Lemma 3] Consider the initial value problems $x'=f_1(t,x)$, $x(0)=x_1$ and $x'=f_2(t,x)$, $x(0)=x_2$ with solutions $\phi_1(t)$ and $\phi_2(t)$ respectively.  If $f_1$ is Lipschitz in $x$ with constant $L$; $\norm{x_2-x_1}\leq \Delta$; and $\norm{f_1(t,x)-f_2(t,x)}\leq M$, then $\norm{\phi_1(t)-\phi_2(t)}\leq (\Delta+M/L)e^{Lt}-M/L$.
\end{lemma}
\begin{proof}  Let $u(t):=\phi_1(t)-\phi_2(t)$.  Then $\norm{u(0)}\leq\Delta$  and 
\begin{align*}
\norm{u'(t)}&=\norm{f_1(t,\phi_1(t))-f_2(t,\phi_2(t))}\\
&\leq \norm{f_1(t,\phi_1(t))-f_1(t,\phi_2(t))}+\norm{f_1(t,\phi_2(t))-f_2(t,\phi_2(t))}\\
&\leq L\norm{u(t)}+M.
\end{align*}
Note that the solution to $\theta'=L\theta+M$ with $\theta(0)=\Delta$ is $\theta(t)=(\Delta+M/L)e^{Lt}-M/L$.  Then by Gronwall's inequality, $\norm{u(t)}\leq (\Delta+M/L)e^{Lt}-M/L$, proving the claim.  \end{proof}

\end{document}